\documentclass[10pt,twoside,english,reqno,a4paper]{amsart}
\usepackage{listings,graphicx,amsmath,varioref,amscd,amssymb,color,bm,stmaryrd,amsthm,amsfonts,graphics,geometry,latexsym,pgf,pst-all} 
\theoremstyle{plain}
\usepackage{esint}
\usepackage{amsthm}
\usepackage{palatino}

\theoremstyle{plain}
\newtheorem{theorem}{Theorem}[section]

\newtheorem{lemma}[theorem]{Lemma}

\usepackage{geometry}
\usepackage[colorlinks=false]{hyperref}

\theoremstyle{definition}
\newtheorem{defin}[theorem]{Definition}

\newtheorem{remark}[theorem]{Remark}
\newtheorem{example}{Example}

\theoremstyle{remark}

\def\bk{\color{black}}

\numberwithin{equation}{section}

\def\into{\int_{\Omega}}
\def\dis{\displaystyle}
\def\supp{\text{\text{supp}}}

\DeclareMathOperator{\sgn}{sgn}
\DeclareMathOperator{\R}{\mathbb{R}}

\newcommand{\car}[1]{\raise1pt\hbox{$\chi$}_{#1}}

\newcommand{\DM }{\mathcal{DM}^\infty }
\def\rn{\mathbb{R}^N}
\def\re{\mathbb{R}}

\begin{document}
\title{Bounded solutions for non-parametric mean curvature problems with nonlinear terms}

\author[D. Giachetti]{Daniela Giachetti}
\author[F. Oliva]{Francescantonio Oliva}
\author[F. Petitta]{Francesco Petitta}

\address{Dipartimento di Scienze di Base e Applicate per l' Ingegneria, Sapienza Universit\`a di Roma, Via Scarpa 16, 00161 Roma, Italia}
\address[Daniela Giachetti]{daniela.giachetti@sbai.uniroma1.it}
\address[Francescantonio Oliva]{francescantonio.oliva@uniroma1.it}
\address[Francesco Petitta]{francesco.petitta@uniroma1.it}

\keywords{prescribed mean curvature, functions of bounded variation, divergence-measure fields, non-parametric minimal surfaces, singular equations} \subjclass[2010]{35J25, 35J60,  35J75, 49Q05, 53A10}

\maketitle

 \begin{abstract}
In this paper we prove existence of nonnegative bounded solutions for  the non-autonomous  prescribed mean curvature problem in non-parametric form on an open bounded domain $\Omega$ of $\R^N$.  The  mean curvature,  that depends on the location of the solution $u$ itself,  is asked to be   of the form $f(x)h(u)$,  where $f$  is  a nonnegative function in $L^{N,\infty}(\Omega)$ and $h:\R^+\mapsto \R^+$ is merely  continuous and possibly unbounded  near zero. As a preparatory tool for our analysis we propose a purely PDE approach to the prescribed mean curvature problem not depending on the solution, i.e. $h\equiv 1$. This part, which  has its own independent interest,  aims to represent    a modern  and   up-to-date account on the subject. Uniqueness is also handled in presence of a decreasing nonlinearity.  The sharpness of the results is highlighted by mean of explicit examples. 
 \end{abstract}
\tableofcontents
\section{Introduction}

Consider an homogeneous Dirichlet boundary value problem related to the  elliptic equation 

\begin{equation}\label{intro}
\dis -\operatorname{div}\left(\frac{D u}{\sqrt{1+|D u|^2}}\right) = f (x)h(u) 
\end{equation}
on an  open bounded domain  $\Omega\subseteq \mathbb{R}^N$; here $f$ is a nonnegative function in $L^{N,\infty}(\Omega)$ and $h:\re^+ \mapsto\re^+$ is  a  continuous function   possibly unbounded near zero.

If $f=0$, then  \eqref{intro} is the well known minimal surface equation
\begin{equation}\label{homi}
\dis \operatorname{div}\left(\frac{D u}{\sqrt{1+|D u|^2}}\right) =0
\end{equation}
   the name deriving from the fact that, for a smooth  function $u$, the involved operator evaluates the mean curvature of the graph of $u$ at each point $(x,u(x))$; due to this fact such an operator is also called {\it non-parametric mean curvature operator}. The unique solvability of Dirichlet problems associated to \eqref{homi} is classical (\cite{js, bdgm, g76,giusti,w} and references therein); 
solutions are known  to exist and to be unique for smooth  boundary data provided $\Omega$ is mean convex; i.e.  $\partial\Omega$ has nonnegative mean curvature. 

Several instances of  (non-parametric) prescribed mean curvature equation of the type
\begin{equation}
\label{pb0}
\begin{cases}
\dis -\operatorname{div}\left(\frac{D u}{\sqrt{1+|D u|^2}}\right) = f(x) & \text{in}\;\Omega,\\
u=0 & \text{on}\;\partial\Omega,
\end{cases}
\end{equation}
have also been considered in the literature both in the case of constant and non-constant data $f$ starting from  \cite{se},  \cite{fg84}, \cite{g}, and \cite{g76,g78} to present a non-complete list. 

Prescribed mean curvature problems  as   \eqref{pb0} are known to formally represent the Euler-Lagrange equation of  a  functional as 
\begin{equation}\label{func}
\mathcal{A}(v)=\int_\Omega \sqrt{1+|\nabla v|^2}  \, dx - \into fv  dx\,, 
\end{equation}
involving the area functional. 

In order to better understand the basis of solvability of problems as in \eqref{pb0} one can formally integrate the equation in there in a smooth sub-domain of $A\subset\Omega$, and using the divergence theorem to obtain the following {\it necessary condition}
\begin{equation}\label{cheeg}
\left|\int_A f(x) dx\right| =\left| \int_{\partial A} \frac{D u}{\sqrt{1+|D u|^2}}\cdot \nu_A ds \right|< {\rm Per}(A)
\end{equation}
where ${\rm Per}(A)$ indicates the perimeter of $A$ and $\nu_A$ is the outer normal unit vector. 
That is, the existence of solutions for \eqref{pb0} is related with a smallness assumption on the datum $f$. This is typical feature of equations arising from functional with linear growth as, for instance,  the one driven by the  $1$-laplacian  (see for instance \cite{CT,KS}).

For constant datum $f=\lambda$ then \eqref{cheeg}  reduces to 
\begin{equation}\label{cheeg2}
|\lambda| < \frac{{\rm Per}(A)}{|A|}\,,
\end{equation}
for any smooth $A\subset\Omega$. The best (positive) constant satisfying \eqref{cheeg2} is known to be the {\it Cheeger constant} of $\Omega$ and this fact, again,  is reminiscent of some known  geometric interpretation for $1$-Laplace type problems (see for instance \cite{K},  the recent \cite{bem},  and the gentle introduction to the subject given in  \cite{leo}).   

 Solvability of  a constant prescribed mean curvature problem  was first noticed to be related to the mean curvature of $\Omega$ in  the celebrated paper \cite{se} by J. Serrin (see also \cite{g76, gt, ber09, csv} and references therein) where a stronger  {\it mean convexity} assumption on $\partial\Omega$ is given, namely (here the datum again $f=\lambda$)
$$
|\lambda|\leq (N-1) \mathcal{K}(y)\, \ \ \text{for all}\ \ y\in\partial\Omega,
$$
where $\mathcal{K}(y)$ is the mean curvature of $\partial\Omega$.

In \cite{g} M. Giaquinta  shows the unique solvability in the space of functions with bounded variation, in a variational sense,  if $f$ is measurable and there exists $\varepsilon_0>0$ such that for every smooth $A\subseteq \Omega$
\begin{equation}\label{giac}
\left|\int_A f(x) \,dx\right|  \leq (1-\varepsilon_0){\rm Per}(A)\,. 
\end{equation}

In \cite{g76} it is shown that 
$$
||f||_{L^N (\Omega)}< N\omega_{N}^{\frac1N} \,,
$$
is a general condition under which \eqref{giac} holds, where $\omega_N$ is the measure of the unit ball of $\rn$. Conditions as in  \eqref{giac} are known as  {\it non-extremal}  conditions (see \cite{ls}) and they represent the necessary and sufficient condition under which the functional \eqref{func} admits a minimum point.
For further considerations on the equality case in \eqref{cheeg} over $\Omega$ (i.e. the  {\it extremal case}) we refer to the very recent paper \cite{lc} and references therein. Although it is out of the scope of the present paper, equations as in \eqref{pb0} have been also considered in the framework of the so-called {\it Mean Curvature Measures} (see \cite{zie}, \cite{dt}, and  \cite{lc});  we want to stress  that our results are consistent and in continuity also with those ones.

\medskip 
Equation \eqref{intro} is a prescribed mean curvature equation with dependence on the location of the graph of the solution itself. 
From the purely theoretical point of view these type of equations naturally appear in many problems of  differential geometry   (see \cite{ber09}).  Concerning this case, i.e.  a right-hand side also depending on the solution we refer to the paper   \cite{m} of M. Miranda in which under monotonicity assumptions on the data the solvability of problems as in \eqref{intro} is dealt again in the "generalized" variational framework of \cite{g76}.  Both interior and global regularity result for solutions of such types of equations we   refer to \cite{Ger, ss, lie, gt, stevsch, po} and references therein. We also mention \cite{am, px, om} and references therein for a more recent account of related results, and \cite{vy} for an interesting application of an (intrinsic) sub- and super-solutions method to these type of problems.  

Both in the autonomous and the non-autonomous case, problems  as in \eqref{intro} arise   in particular  in the study of  combustible gas dynamics (see \cite{yy} and references therein) as well as in  surfaces capillary problem as pendant liquid drops (\cite{finn74,cf1,cf2, fg84}) and, as a curiosity, also in design of water-walking devices (\cite{ww}, see also \cite{lc}).

\medskip

The aim of this paper is to provide a sharp description of homogeneous boundary value problems involving  \eqref{intro} with (possibly weak) Lebesgue data and a purely PDE's approach. 

To better emphasize the main difficulties in treating such  problems we decide to work  first in the context of positive Lebesgue data, namely $f\in L^{N}(\Omega)$ with $f>0$. To the extensions to the cases of both nonnegative data and  of data in the (sharp) scale  of Marcinkiewicz spaces $f\in L^{N,\infty}(\Omega)$, which are nowadays quite customary,   we dedicate  (resp.) Section \ref{sec:fnonnegative} and Section \ref{52} below.

\medskip 

After the preparatory Section \ref{due} in which we set the basic machinery on $BV$ spaces (the natural space in which these problems are well settled), measure divergence vector fields and 
Anzellotti-Chen-Frid type theory of pairings, Section \ref{sec3} is devoted to present in a self-contained and up-to-date way the existence of bounded solutions to problem \eqref{pb0} that is  what is needed to our further  aims. 
The core of the paper is the content of Section \ref{4} in which under suitable smallness assumptions on the data we prove existence and (once expected) uniqueness of bounded solutions for homogeneaous boundary value problems associated to \eqref{intro}. As we already mentioned  Section \ref{5} is devoted to the extension of the previous results to the case of nonnegative data in $L^{N,\infty}(\Omega)$. The optimality of the smallness assumption on the data will be also discussed by mean of  explicit examples of solutions in Section \ref{533}. 

\bk 

\subsection*{Notation}

Here  $\Omega$  will always be an open bounded subset of $\R^N$ ($N\ge 2$) with Lipschitz boundary.
We denote by $\mathcal H^{N-1}(\partial E)$ (also  ${\rm Per}(E)$ somewhere) the $(N - 1)$-dimensional Hausdorff measure of the boundary of a set $E$,   while $|E|$ stands for its $N$-dimensional  Lebesgue measure.

 We denote by $\chi_{E}$ the characteristic function of a set $E$. For a fixed $k>0$, we use the truncation functions $T_{k}:\R\to\R$ and $G_{k}:\R\to\R$ defined, resp.,  by
\begin{align*}
T_k(s):=&\max (-k,\min (s,k))\ \ \text{\rm and} \ \ G_k(s):=s- T_k(s).
\end{align*}
We will also made use of  the following auxiliary function defined for $s\in\re^+$
\begin{align}\label{Vdelta}
\displaystyle
V_{\delta}(s):=
\begin{cases}
1 \ \ &  0\le s\le \delta, \\
\displaystyle\frac{2\delta-s}{\delta} \ \ &\delta <s< 2\delta, \\
0 \ \ &s\ge 2\delta.
\end{cases}
\end{align}
We denote by $\mathcal{S}_p$ the best constant in the Sobolev inequality ($1\leq p< N$), that  is
\begin{equation}\label{sob}
||v||_{L^{p^*}(\Omega)} \le \mathcal{S}_p ||v||_{W^{1,p}_0(\Omega)}, \ \ \forall v \in W^{1,p}_0(\Omega),\end{equation}
where $p^*=\frac{Np}{N-p}$. It is known that 
$$\displaystyle \lim_{p\to 1^{+}} \mathcal{S}_p = \mathcal{S}_1=(N\omega_{N}^{\frac{1}{N}})^{-1}\,,$$
where   	$\omega_{N}$ is the volume of the unit sphere of $\rn$.

\medskip 

If not otherwise specified, we will denote by $C$ several positive constants whose value may change from line to line and, sometimes, on the same line. These values will only depend on the data but they will never depend on the indexes of the sequences we will gradually introduce. Let us explicitly mention that we will not relabel an extracted compact subsequence.

Finally for the sake of simplicity, where no ambiguity is possible,  we  use the following notation for the Lebesgue integral of a function $f$ 
$$
 \int_\Omega f:=\int_\Omega f(x)\ dx\,.
$$

\section{Basic tools}  \label{due}
\subsection{Basics on $BV$ spaces and the area integral} 
We refer to \cite{afp} for a complete account on $BV$-spaces and, for the sake of brevity,  for further standard notations not mentioned here.

Let us define  
$$BV(\Omega):=\{ u\in L^1(\Omega) : Du \in \mathcal{M}(\Omega)^N \}.$$ 
By $Du \in \mathcal{M}(\Omega)^N$ we mean that  each distributional partial derivative of $u$ is a bounded Radon measure; the total variation of the vector valued measure $D u$ is given by
$$\displaystyle |D u| = \sup\left\{\int_\Omega u\sum_{i=1}^{N}  \operatorname{\frac{\partial \phi_i}{\partial  x_i}}, \ \phi_i \in C^1_0(\Omega, \mathbb{R}), \ |\phi_i|\le 1, \forall i=1,...,N\right\}.$$

We underline that the $BV(\Omega)$ space endowed with the norm  
$$\displaystyle ||u||_{BV(\Omega)}=\int_{\partial\Omega}
|u|\, d\mathcal H^{N-1}+ \int_\Omega|Du|,$$
is a Banach space.
By $BV_{\rm loc}(\Omega)$ we mean the space of functions in $BV(\omega)$ for every open set $\omega \subset\subset\Omega$.
\\
For a given Radon measure $\mu$ we will frequently use that it can be uniquely decomposed as 
$\mu=\mu^a+\mu^s$
where $\mu^a$ is absolutely continuous with respect to the Lebesgue measure while $\mu^s$ is concentrated on a set of zero Lebesgue measure.

If $u\in BV(\Omega)$ one can give sense to the measure $\sqrt{1+|Du|^2}$ by defining it as  
$$\displaystyle  \sqrt{1+ | D u|^2} (E)= \sup\left\{\int_E \phi_{N+1} - \int_E u\sum_{i=1}^{N}  \operatorname{\frac{\partial \phi_i}{\partial x_i}}, \ \phi_i \in C^1_0(\Omega, \mathbb{R}), \ |\phi_i|\le 1, \forall i =1,...,N+1\right\}\,,$$
for any Borel set $E\subseteq \Omega$. 
We will frequently write  
$$\int_\Omega \sqrt{1+ | D u|^2}$$ 
meaning the total variation of the  $\mathbb{R}^{N+1}$-valued measure which formally represents $(\mathcal{L}^N, Du)$. Indeed, if $u$ is smooth,  then 
$$|(\mathcal{L}^N, \nabla u)| (\Omega)=\int_\Omega \sqrt{1+ |\nabla u|^2}$$
gives 
the area of the graph of $u$. 
In general, it simply follows from the decomposition in absolutely continuous and singular part with respect to the Lebesgue measure that one can write   
$$
	\sqrt{1+ | D u|^2}= \sqrt{1+ | D^a u|^2}\mathcal{L}^N+|D^s u| \,.
$$
In the sequel we will use the following semicontinuity classical results; firstly, the  functional $$J_1(v)=\int_\Omega \sqrt{1+ | D v|^2}\varphi + \int_{\partial\Omega} |v|\varphi \, d\mathcal H^{N-1} , \ \ {\text{for all} \ 0\le \varphi \in C^1(\overline{\Omega})}$$ 
is lower semicontinuous in $BV(\Omega)$ with respect to the $L^1(\Omega)$ convergence. On the other hand  the functional $$J_2(v) = \int_\Omega \sqrt{1-|v|^2}\varphi \ \ {\text{for all} \ 0\le \varphi \in C^1(\overline{\Omega})}$$ is weakly upper semicontinuous with respect to the $L^1(\Omega)$ convergence (see Corollary $3.9$ of \cite{brezis}).

\subsection{The Anzellotti-Chen-Frid theory} In order to be self-contained we  summarize the $L^\infty$-divergence-measure vector fields theory due to \cite{An} and \cite{CF}. We denote by 
$$\DM(\Omega):=\{ z\in L^\infty(\Omega)^N : \operatorname{div}z \in \mathcal{M}(\Omega) \},$$
and by $\DM_{\rm loc}(\Omega)$ its local version, namely the space of bounded vector field $z$ with $\operatorname{div}z \in \mathcal{M}_{\rm loc}(\Omega)$.

In \cite{An}, if $v \in BV(\Omega) \cap C(\Omega)$, the following distribution $(z,Dv): C^1_c(\Omega)\to \mathbb{R}$ is considered: 
\begin{equation}\label{dist1}
\langle(z,Dv),\varphi\rangle:=-\int_\Omega v^*\varphi\operatorname{div}z-\int_\Omega
vz\cdot\nabla\varphi,\quad \varphi\in C_c^1(\Omega),
\end{equation}
where $v^*$ is the precise representative for $v$. 
In \cite{mst2} and \cite{C} the authors prove that $(z, Dv)$ is well defined if $z\in \DM(\Omega)$ and $v\in BV(\Omega)\cap L^\infty(\Omega)$ since one can show that $v^*\in L^\infty(\Omega,\operatorname{div}z)$. Also observe that, if $\operatorname{div}z$ is a function then $v^*$ can be substituted by $v$ in \eqref{dist1}.

Moreover in \cite{dgs} it is  shown that \eqref{dist1} is well posed if $z\in \DM_{\rm loc}(\Omega)$ and $v\in BV_{\rm loc}(\Omega)\cap L^1_{\rm loc}(\Omega, \operatorname{div}z)$; it holds
\begin{equation*}\label{finitetotal}
|\langle   (z, Dv), \varphi\rangle| \le ||\varphi||_{L^{\infty}(U) } ||z||_{L^\infty(U)^N} \int_{U} |Dv|\,,
\end{equation*}
for all open set $U \subset\subset \Omega$ and for all $\varphi\in C_c^1(U)$. One has
\begin{equation*}\label{finitetotal1}
\left| \int_B (z, Dv) \right|  \le  \int_B \left|(z, Dv)\right| \le  ||z||_{L^\infty(U)^N} \int_{B} |Dv|\,,
\end{equation*}
for all Borel sets $B$ and for all open sets $U$ such that $B\subset U \subset \Omega$.
We recall that  every $z \in \mathcal{DM}^{\infty}(\Omega)$ possesses a weak trace on $\partial \Omega$ of its normal component which is denoted by
$[z, \nu]$, where $\nu(x)$ is the outward normal unit vector defined for $\mathcal H^{N-1}$-almost every $x\in\partial\Omega$ (see  \cite{An}). Moreover, it holds
\begin{equation*}\label{des1}
||[z,\nu]||_{L^\infty(\partial\Omega)}\le ||z||_{L^\infty(\Omega)^N},
\end{equation*}
and also,  if $z \in \mathcal{DM}^{\infty}(\Omega)$ and $v\in BV(\Omega)\cap L^\infty(\Omega)$, that
\begin{equation}\label{des2}
v[z,\nu]=[vz,\nu],
\end{equation}
(see \cite{C}).\\
Furthermore, if $z\in \DM_{\rm loc}(\Omega)$ and $v\in BV(\Omega)\cap L^\infty(\Omega)$ such that $v^*\in L^1(\Omega,\operatorname{div}z)$,  then $vz\in \DM(\Omega)$ and a weak trace can be defined as well as the following Green formula (\cite{dgs}):
\begin{lemma}\label{21}
	Let $z\in \DM_{\rm loc}(\Omega)$ and $v \in BV(\Omega)\cap L^{\infty}(\Omega)$ such that $v^*\in L^1(\Omega, \operatorname{div}z)$ then it holds
	\begin{equation}\label{green}
		\int_{\Omega} v^* \operatorname{div}z + \int_{\Omega} (z, Dv) = \int_{\partial \Omega} [vz, \nu] \ d\mathcal H^{N-1}.
	\end{equation}
\end{lemma}	
Formula \eqref{green} continues to hold if $z\in L^\infty(\Omega)^N$ such that $\operatorname{div}z \in L^N(\Omega)$ and $v\in BV(\Omega)$.

Finally, we also  recall the following technical result due to \cite[Theorem 2.4]{An}. \begin{lemma}\label{lemanzas}
Let $u\in BV(\Omega)$ and let $z\in \DM(\Omega)$ such that $u^*\in L^1(\Omega, \operatorname{div}z)$ then
$$(z, D u)^a = z \cdot D^a u.$$ 
\end{lemma}

\medskip 

\section{The prescribed mean curvature case with $f\in L^N(\Omega)$}  \label{sec3}
 In this section we  deal with existence of weak solutions to the following problem:  
\begin{equation}
\label{pb}
\begin{cases}
\dis -\operatorname{div}\left(\frac{D u}{\sqrt{1+|D u|^2}}\right) = f & \text{in}\;\Omega,\\
u=0 & \text{on}\;\partial\Omega,
\end{cases}
\end{equation}
with a datum $f$ belonging  to $L^N(\Omega)$. Though it can be viewed as preparatory  for the general case considered later,  most  of the results of the present section are new both in the form  and in their proofs. Some of the results, those reminiscent  of the classical ones, are  recasted in a up-to-date fashion. We do not assume any sign condition on $f$. 
As we will see, a suitable approximation argument will take us to a $BV$-solution. Let us be  precise in what we mean:

\begin{defin}
	\label{weakdef}
	A function $u\in BV(\Omega)$ is a solution to problem \eqref{pb} if there exists $z\in \mathcal{D}\mathcal{M}^\infty(\Omega)$ with $||z||_{L^\infty(\Omega)^N}\le 1$ such that
	\begin{align}
	&-\operatorname{div}z = f \ \ \text{in}\ \ \mathcal{D}'(\Omega), \label{def_distrp=1}
	\\
	&(z,Du)=\sqrt{1+|Du|^2} - \sqrt{1-|z|^2} \label{def_zp=1} \ \ \ \ \text{as measures in } \Omega,
	\\
	&u(\sgn{u} + [z,\nu])(x)=0 \label{def_bordop=1}\ \ \ \text{for  $\mathcal{H}^{N-1}$-a.e. } x \in \partial\Omega.
	\end{align}
\end{defin}
\begin{remark}\label{remdef}
	Let us spend a few words on Definition \ref{weakdef} and in particular on the request given by \eqref{def_zp=1} that is a  weak way to interpret the ratio between the two measures $D u$ and $\sqrt{1+|D u|^2}$: if $u$ is smooth and $z=\frac{\nabla u}{\sqrt{1+|\nabla u|^2}}$ then one has
	$$
	(z, \nabla u)= z\cdot \nabla u=\frac{|\nabla u|^2}{\sqrt{1+|\nabla u|^2}} \,,
	$$   
	which is exactly  the right-hand of \eqref{def_zp=1}.

Then  it is easy to see that \eqref{def_zp=1} can be equivalently recast by  requiring that both 
	\begin{equation}\label{remdef1}
		z\cdot D^a u = \sqrt{1+ |D^a u|^2} - \sqrt{1-|z|^2}
	\end{equation}
	and
	\begin{equation*} \label{remdef2}
		(z,Du)^s = |D^s u|
	\end{equation*}
	hold. Indeed, using Lemma \ref{lemanzas} and the fact that $f\in  L^N(\Omega)$, one has that  
	$$(z,D u)^a  = z\cdot D^a u.$$ 
	We stress  that, in contrast  with  other cases of flux-limited diffusion operators (e.g. the $1$-laplacian or the transparent media one, \cite{ACM, ABCM, GMP}),  here  the vector field $z$ is uniquely determined by  \eqref{remdef1}, which gives 

$$
z= \frac{D^a u}{\sqrt{1+|D^a u |^2}}. 
$$

 Finally condition \eqref{def_bordop=1} is a nowdays standard way to give meaning to the homogeneous Dirichlet boundary datum. It is well known, in fact,  that $BV$ solutions to problems involving  such type of operators (e.g. the  $1$-laplacian)  do not necessarily assume the boundary datum pointwise.  \eqref{def_bordop=1}  roughly asserts that either $u$ has zero trace or the weak trace of the normal component of $z$ has least possible slope at the boundary.
\end{remark}

Let us state the main result of this section which gives existence of solutions to \eqref{pb} under a smallness condition on the datum $f$.

\begin{theorem}\label{teomain}
	Let $f\in L^N(\Omega)$ such that \begin{equation}\label{smallness}\displaystyle ||f||_{L^N(\Omega)}<\frac{1}{\mathcal{S}_1}\,.\end{equation} Then there exists a bounded solution to problem \eqref{pb}.
\end{theorem}

\begin{remark}
As we already said, assumption \eqref{smallness} is in  some sense necessary in order to get a solution, we refer to  Remark \ref{giusti} below for further comments on that and also to Example \ref{53} in Section \ref{533}.  

We also want to stress that a general uniqueness result, in contrast with the regular case (as in \cite{gt}), is not expected in this generality.  In fact, in \cite[Theorem 9.1]{lc},  the authors prove  a uniqueness result in the class of  continuous functions providing an example that highlights  a non-uniqueness phenomenon for solutions with non-empty jump part. 
\end{remark}

The proof of  Theorem \ref{teomain} will be obtained through  approximation with the $p$-growth ($p>1$)  problems
\begin{equation}
	\label{pbp}
	\begin{cases}
		\dis \dis -\operatorname{div}\left(\frac{\nabla u_p}{\sqrt{1+|\nabla  u_p|^2}}\right) - (p-1)\operatorname{div}\left(|\nabla u_p|^{p-2}\nabla u_p\right)  = f & \text{in}\;\Omega,\\
		u_p=0 & \text{on}\;\partial\Omega,
	\end{cases}
\end{equation}
a solution of \eqref{pbp} being a function $u_p\in W^{1,p}_0(\Omega)$ such that 
\begin{equation}\label{pbpw}
\into \frac{\nabla u_p}{\sqrt{1+|\nabla  u_p|^2}}\cdot \nabla v + (p-1)\into |\nabla u_p|^{p-2}\nabla u_p\cdot \nabla v = \into fv\,,
\end{equation}
for any $v\in W^{1,p}_0(\Omega)$.

The existence  of a solution  $u_p \in W^{1,p}_0(\Omega)\cap L^\infty(\Omega)$  of \eqref{pbp} follows by  standard monotonicity arguments (\cite{ll}).

\medskip

We accomplish the proof of  Theorem \ref{teomain} by splitting it into a few  steps. We start showing some estimates on $u_p$ which are independent of $p\sim 1^+$. In fact,  we  recall that the aim is taking $p\to 1^+$ in \eqref{pbp} so that,  in the following, estimates independent on $p$ are tacitly meant as  {\it  there exists some $p_0>1$ such that the estimate is uniform in the range $1<p\leq p_0$}.

\medskip

The main needed estimates are collected in the following: 
\begin{lemma}\label{lemma_stime}
Let $f\in L^N(\Omega)$ such that \eqref{smallness} holds and let $u_p$ be a solution to \eqref{pbp}. Then $u_p$ is bounded in $BV(\Omega)\cap L^\infty(\Omega)$ (with respect to $p$) and there exists $u\in BV(\Omega)\cap L^\infty(\Omega)$ such that, up to subsequences, $u_p$ converges to $u$ in $L^q(\Omega)$ for every $q<\infty$, weak$^*$ in $L^\infty(\Omega)$, and  $\nabla u_p$ converges to $D u$ weak$^*$ as measures as $p\to 1^+$. Moreover it holds
\begin{equation}\label{stimatermineenergia}
	(p-1)\int_\Omega |\nabla u_p|^p\le C,
\end{equation}
for some constant $C$ independent of $p$.
\end{lemma}
\begin{proof}
	Let us show that $u_p$ are uniformly  bounded in $L^\infty(\Omega)$. We take $G_k(u_p)$ ($k>0$) as a test function in \eqref{pbpw} and we use H\"older's inequality,  yielding to
	\begin{equation}\begin{aligned}\label{stimeapriori4}
		\int_\Omega \frac{|\nabla G_k(u_p)|^2}{\sqrt{1+|\nabla G_k(u_p)|^2}} & \le  \int_\Omega f G_k(u_p) 
		\le ||f||_{L^N(\Omega)} ||G_k(u_p)||_{L^{\frac{N}{N-1}}(\Omega)} \\  &\stackrel{\eqref{sob}}{\le} ||f||_{L^N(\Omega)} \mathcal{S}_1 \int_\Omega |\nabla G_k(u_p)|,
	\end{aligned}\end{equation}
 after getting rid of the nonnegative second term.
 
Now let us focus on the  first term of \eqref{stimeapriori4}; one has
	\begin{equation}\label{stimeapriori5}
		\int_\Omega \frac{|\nabla G_k(u_p)|^2}{\sqrt{1+|\nabla G_k(u_p)|^2}}  = \int_{A_k} {\sqrt{1+|\nabla G_k(u_p)|^2}}  - \int_{A_k} \frac{1}{\sqrt{1+|\nabla G_k(u_p)|^2}}  \ge \int_\Omega |\nabla G_k(u_p)| - |A_k|,
	\end{equation}
	where $A_k:=\{x\in\Omega: |u_p(x)| > k\}$. Then, using \eqref{stimeapriori5} in \eqref{stimeapriori4} and thanks to \eqref{smallness}, one gets
	\begin{equation*}\label{stimeapriori6}
		\int_\Omega |\nabla G_k(u_p)| \le \frac{|A_k|}{1-||f||_{L^N(\Omega)} \mathcal{S}_1}.
	\end{equation*}
	The Sobolev and  the H\"older inequalities together with  the previous imply that
	\begin{equation*}
		\int_\Omega |G_k(u_p)| \le \frac{|A_k|^{1+\frac{1}{N}}\mathcal{S}_1}{1-||f||_{L^N(\Omega)} \mathcal{S}_1}.
	\end{equation*}
	In particular,  for any $h>k>0$, one has 
	\begin{equation}\label{stimeapriori7}
	|A_h| \le \frac{|A_k|^{1+\frac{1}{N}}\mathcal{S}_1}{(h-k)(1-\mathcal{S}_1 ||f||_{L^N(\Omega)}) }
	\end{equation}	
	which allows to apply the classical Stampacchia argument (see \cite{Stampacchia}) in order to deduce that $||u_p||_{L^\infty(\Omega)}\le M$ where, we stress it,  $M>0$ does not depend on $p$ as the right-hand of \eqref{stimeapriori7} does not.	
		
	\medskip

	Now we turn  on proving that $u_p$ is bounded in $BV(\Omega)$. We plug  $u_p$ as  test in  \eqref{pbpw}, yielding to 
	\begin{equation}\label{stimeapriori1}
		\int_\Omega \frac{|\nabla u_p|^2}{\sqrt{1+|\nabla u_p|^2}} + (p-1)\int_\Omega |\nabla u_p|^p = \int_\Omega f u_p \le M\int_\Omega f\,.
	\end{equation}
	
	For the left-hand of \eqref{stimeapriori1} we reason as for the first part of the proof
	\begin{equation}\label{stimeapriori1bis}
	\int_\Omega \frac{|\nabla u_p|^2}{\sqrt{1+|\nabla u_p|^2}}  = \int_\Omega {\sqrt{1+|\nabla u_p|^2}}  - \int_\Omega \frac{1}{\sqrt{1+|\nabla u_p|^2}}  \ge \int_\Omega |\nabla u_p| - |\Omega|.
	\end{equation}	
	Thus, collecting \eqref{stimeapriori1} and \eqref{stimeapriori1bis} and applying the Young inequality one gets 
\begin{equation*}\label{stimeapriori2}
\int_{\Omega} |\nabla u_p| + (p-1)\int_\Omega |\nabla u_p|^p \le M\int_\Omega f + |\Omega|,
\end{equation*}
which gives \eqref{stimatermineenergia} and also implies the boundedness in $BV(\Omega)$ of $u_p$. 

\medskip 
The $BV$ estimate  (joint with the $L^\infty$ one)  for $u_p$ allows to apply standard compactness arguments; so there exists a function $u$ such that, up to subsequences, $u_p$ converges to $u$ in $L^q(\Omega)$ for every $q<\infty$, weak$^*$ in $L^\infty(\Omega)$,  and such that $\nabla u_p$ converges to $D u$  weak$^*$ as measures as $p\to 1^+$. This concludes the proof.
\end{proof}

 Next lemma concerns the identification and the role of the vector field $z$.

\begin{lemma}\label{lemma_esistenzaz}
	Under the assumptions of Lemma \ref{lemma_stime} there exists $z\in \DM(\Omega)$ such that 
	\begin{equation}\label{esistenzaz1}
	-\operatorname{div}z= f  \ \ \ \text{in  } \mathcal{D'}(\Omega),
	\end{equation}
	and
	\begin{equation}\label{esistenzaz2}
		(z,Du)=\sqrt{1+|Du|^2} - \sqrt{1-|z|^2} \ \ \text{as measures in } \Omega,
	\end{equation}
	where $u$ is the function found in Lemma \ref{lemma_stime}.
\end{lemma}
\begin{proof}
Since ${|\nabla u_p|}{({1+|\nabla  u_p|^2})^{-\frac12}}\le 1$ there exists a bounded vector field $z$ such that ${\nabla u_p}{({1+|\nabla  u_p|^2})^{-\frac12}}$ converges to $z$ weak$^*$ in $L^\infty(\Omega)^N$ as $p\to 1^+$. Moreover by weak lower semicontinuity of the norm, one gets that $||z||_{L^\infty(\Omega)^N} \le 1$.

Now let us take $\varphi \in C^1_c(\Omega)$ as a test function in \eqref{pbpw} and let us  take $p\to 1^+$, one has
\begin{equation}\label{phitest}
\into z\cdot \nabla \varphi +\lim_{p\to1^+} (p-1)\into |\nabla u_p|^{p-2}\nabla u_p\cdot \nabla \varphi = \into f\varphi\,.
\end{equation}

	Let us now observe that
	\begin{equation}\label{termineenergiazero}
	\begin{aligned}
		(p-1)\left|\int_\Omega |\nabla u_p|^{p-2}\nabla u_p\cdot\nabla \varphi \right| &\le (p-1) \left(\int_\Omega |\nabla u_p|^p\right)^{\frac{p-1}{p}} \left(\int_\Omega |\nabla \varphi|^p\right)^{\frac{1}{p}} 
		\\
		&\le (p-1)^{\frac{1}{p}} ||\nabla \varphi||_{L^\infty(\Omega)^N}|\Omega|^{\frac{1}{p}} \left((p-1)\int_\Omega |\nabla u_p|^p\right)^{\frac{p-1}{p}}
		\\
		&\stackrel{\eqref{stimatermineenergia}}{\le}(p-1)^{\frac{1}{p}} ||\nabla \varphi||_{L^\infty(\Omega)^N}|\Omega|^{\frac{1}{p}} C^{\frac{p-1}{p}},
	\end{aligned}
\end{equation}
which gives that the second term in \eqref{phitest} vanishes. This implies  that  \eqref{esistenzaz1} holds and that $z\in \DM(\Omega)$.

\medskip

	It is worth mentioning for later purposes that, since $u\in L^\infty(\Omega)$ and $f\in L^N(\Omega)$, one can easily check that 
		\begin{equation}\label{moltperu}
	-u\operatorname{div}z =  f u  \ \ \text{  in  } \mathcal{D}'(\Omega).
\end{equation}
\medskip
	Now, recalling Remark \ref{remdef}, in order to prove \eqref{esistenzaz2}, it suffices to  show both
		\begin{equation}\label{zducontinua}
		z\cdot D^a u = \sqrt{1+ |D^a u|^2} + \sqrt{1-|z|^2}
	\end{equation}
	and
	\begin{equation} \label{zdusingolare}
		(z,Du)^s = |D^s u|.
	\end{equation}

\medskip
	
We take $v=u_p\varphi$ in  \eqref{pbpw}  where $\varphi \in C^1_c(\Omega)$ is a nonnegative function, yielding to

\begin{equation}\begin{aligned}\label{lemmaz1}	
&\int_{\Omega} \frac{|\nabla u_p|^{2}\varphi}{\sqrt{1 +|\nabla u_p|^2}} + \int_{\Omega} \frac{\nabla u_p \cdot \nabla \varphi u_p}{\sqrt{1 +|\nabla u_p|^2}} + (p-1)\int_\Omega |\nabla u_p|^p\varphi
\\
& + (p-1)\int_\Omega |\nabla u_p|^{p-2}\nabla u_p\cdot \nabla \varphi u_p = \int_{\Omega}  f u_p \varphi.	
\end{aligned}
\end{equation}

We can write the first term on the left-hand of the previous as
\begin{equation}
	\label{firstterm}
	\int_{\Omega} \frac{|\nabla u_p|^{2}\varphi}{\sqrt{1 +|\nabla u_p|^2}} = \int_{\Omega} \sqrt{1 +|\nabla u_p|^2}\varphi - \int_{\Omega} \sqrt{1-\frac{|\nabla u_p|^{2}}{1 +|\nabla u_p|^2}}\varphi,
\end{equation}
and we drop  the nonnegative third term in \eqref{lemmaz1}. This takes to
\begin{equation}\begin{aligned}\label{lemmaz2}	
	&\int_{\Omega} \sqrt{1 +|\nabla u_p|^2}\varphi - \int_{\Omega} \sqrt{1-\frac{|\nabla u_p|^{2}}{1 +|\nabla u_p|^2}}\varphi + \int_{\Omega} \frac{\nabla u_p \cdot \nabla \varphi u_p}{\sqrt{1 +|\nabla u_p|^2}} 
	\\
	&+  (p-1)\int_\Omega |\nabla u_p|^{p-2}\nabla u_p\cdot \nabla \varphi u_p \le \int_{\Omega}  f u_p \varphi.	
\end{aligned}\end{equation}
Now observe that, as $p\to 1^+$, the first term is lower semicontinuous with respect to the $L^1$ convergence. As we already mentioned, using  Corollary $3.9$ of \cite{brezis}, one can  deduce that the second term of \eqref{lemmaz2}	 is weakly lower semicontinuous with respect to the $L^1$ convergence (recall that   $\nabla u_p(1 +|\nabla u_p|^2)^{-\frac{1}{2}}$ converges to $z$ weak$^*$ in $L^\infty(\Omega)^N$). 
Moreover, the third term on the left-hand of \eqref{lemmaz2}  passes to the limit by  the weak$^*$ convergence of $\nabla u_p(1 +|\nabla u_p|^2)^{-\frac{1}{2}}$ to $z$ in $L^\infty(\Omega)^N$ together with the strong convergence of $u_p$ in $L^q(\Omega)$ for any $q<\infty$ as $p\to 1^+$. 
The convergence of $u_p$ in $L^q(\Omega)$ for any $q<\infty$ also allows to pass to the limit the term on the right-hand of \eqref{lemmaz2}.
\\
It remains to estimate  the fourth term on the left-hand side; indeed, as $u_p$ is bounded in $L^\infty(\Omega)$
	\begin{equation*}\label{termineenergiazero2}
	\begin{aligned}
		(p-1)\left|\int_\Omega |\nabla u_p|^{p-2}\nabla u_p\cdot\nabla \varphi u_p \right| \le ||u_p||_{L^\infty(\Omega)}(p-1)\left|\int_\Omega |\nabla u_p|^{p-2}\nabla u_p\cdot\nabla \varphi \right| 
		\end{aligned}
\end{equation*}
and the right-hand of the previous goes to zero as $p\to1^+$ as in  \eqref{termineenergiazero}. 

Therefore, we have  that  
\begin{equation*}	
\int_{\Omega} \sqrt{1 +|D u|^2}\varphi - \int_\Omega \sqrt{1-|z|^2}\varphi \le - \int_{\Omega} uz\cdot \nabla \varphi + \int_{\Omega}  f u \varphi  \overset{\eqref{moltperu}}{=} -\int_{\Omega} uz\cdot \nabla \varphi-\int_\Omega u\operatorname{div}z \varphi,	
\end{equation*}
and by \eqref{dist1}, one has
\begin{equation}\label{zdu0}	
\int_{\Omega} \sqrt{1 +|D u|^2}\varphi - \int_\Omega \sqrt{1-|z|^2}\varphi \le \int_\Omega (z, Du)\varphi,	\ \ \forall  \varphi \in C^1_c(\Omega), \ \varphi \ge 0.
\end{equation}
Since $\operatorname{div}z \in L^N(\Omega)$ and thanks to Lemma \ref{lemanzas}, inequality \eqref{zdu0} implies that almost everywhere in $\Omega$ 
\begin{equation*}\label{zdu1}
	z\cdot D^a u \ge \sqrt{1 +|D^a u|^2} - \sqrt{1-|z|^2}.
\end{equation*}
The reverse inequality is purely algebraic; indeed, for any $\xi \in \re^N$ it holds
$$z\cdot \xi +\sqrt{1-|z|^2}\le \sqrt{1+|\xi|^2}.$$

This proves the validity of \eqref{zducontinua}. Concerning  \eqref{zdusingolare}, 
since $||z||_{L^\infty(\Omega)^N}\le 1$ then
$$(z,Du)^s \le |D u|^s= |D^s u|,$$
as measures in $\Omega$.
The reverse inequality simply follows from \eqref{zdu0} by restricting on the singular parts of the measures. This concludes the proof.
\end{proof}

Let us now show that the boundary condition \eqref{def_bordop=1} holds.

\begin{lemma}\label{lemma_datoalbordo}
	Under the assumptions of Lemma \ref{lemma_stime} it holds
	$$u(\sgn{u} + [z,\nu])(x)=0 \ \text{for  $\mathcal{H}^{N-1}$-a.e. } x \in \partial\Omega,$$
	where $u$ and $z$ are, resp.,  the function and the vector field found, resp., in lemmata \ref{lemma_stime} and \ref{lemma_esistenzaz}.	
\end{lemma}
\begin{proof}
	One tests \eqref{pbpw}  with  $u_p$  and, recalling that $u_p$ has zero Sobolev trace on $\partial\Omega$, one gets 
	\begin{equation*}
	\int_{\Omega} \frac{|\nabla u_p|^2}{\sqrt{1+|\nabla u_p|^2}} + \int_{\partial \Omega} |u_p| d\mathcal{H}^{N-1}  \le  \int_{\Omega}  fu_p.	
	\end{equation*}		
	By elementary manipulations  as  done in \eqref {firstterm} (with $\varphi =1$) one has 
			\begin{equation*}\label{boundary0}
		\int_{\Omega} \sqrt{1+|\nabla u_p|^2} - \int_{\Omega} \sqrt{1-\frac{|\nabla u_p|^2}{1+|\nabla u_p|^2}} + \int_{\partial \Omega}|u_p| d\mathcal{H}^{N-1} \le \int_{\Omega}  f u_p.	
	\end{equation*}
	Similarly to what we have done in Lemma \ref{lemma_esistenzaz}, one can take the liminf for the left-hand of the previous and use the {weak lower and upper semicontinuity} in order to get
	\begin{equation*}\label{boundary1}
	\int_{\Omega} \sqrt{1+|D u|^2} - \int_{\Omega} \sqrt{1-|z|^2} + \int_{\partial \Omega}|u| d\mathcal{H}^{N-1} \le  \int_{\Omega}  f u.	
	\end{equation*}
	Now recall \eqref{moltperu}, i.e.   $$-u\operatorname{div}z=fu\ \ \ \text{in}\ \  \Omega.$$ 
	Finally observe	\begin{equation*}\label{bordo1}
	\begin{aligned}
	\int_{\Omega} \sqrt{1+|Du|^2} - \int_{\Omega} \sqrt{1-|z|^2} + \int_{\partial \Omega}|u| d\mathcal{H}^{N-1}  & \stackrel{\eqref{moltperu}}{\le}  -\int_\Omega u\operatorname{div}z 
	\\
	&\stackrel{\eqref{green}}{=} \int_{\Omega}(z,Du) - \int_{\partial \Omega} u[ z,\nu]d\mathcal{H}^{N-1},
	\end{aligned}	
	\end{equation*}
	where in the last step we also used \eqref{des2}.  
This concludes   the proof  as   \eqref{esistenzaz2} is in force and recalling that $|[z,\nu]|\leq 1$.
\end{proof}
As we said the  proof of Theorem \ref{teomain} simply follows by gathering together the previous lemmata.
\begin{proof}[Proof of Theorem \ref{teomain}]
	Let $u_p$ be a solution to \eqref{pbp} then the proof is a consequence of Lemmas \ref{lemma_stime}, \ref{lemma_esistenzaz} and \ref{lemma_datoalbordo}.
\end{proof}

\begin{remark}
\label{giusti}
We stress that assumption \eqref{smallness} is essentially equivalent to the necessary and sufficient condition given in \cite{g} in order to get a minimum point for the associated functional, i.e. there exists $\varepsilon_0>0$ such that for every $A\subseteq \Omega$
\begin{equation}\label{giacr}
\left|\int_A f(x) \,dx\right|  \leq (1-\varepsilon_0){\rm Per}(A) \,. 
\end{equation}
In fact, as already mentioned, condition \eqref{smallness} implies \eqref{giacr} (see \cite{g76}); on the other hand if, for simplicity, we consider  a constant datum $f=\lambda$ satisfying  \eqref{giacr} on a ball $B_R$, then
$$
|\lambda | \leq (1-\varepsilon_0)\frac{{\rm Per}(B_R)}{|B_R|}<\frac{N}{R}\,,
$$
that implies that  \eqref{smallness} (i.e. $|\lambda | < \frac{N}{R}$ in this case) holds true . 

\medskip 
Although it is not effortless, in general, proving the equivalence among variational and weak solutions, we   emphasize that  the smallness assumption \eqref{smallness} is, in some sense,  sharp also in our  framework.
\medskip 

 In fact, let $0<r<R$ and consider $u$ to be  a solution of 
$$
\begin{cases}
\dis -\operatorname{div}\left(\frac{D u}{\sqrt{1+|D u|^2}}\right) = \lambda & \text{in}\; B_R,\\
u=0 & \text{on}\;\partial B_R,
\end{cases}$$
in the sense of Definition \ref{weakdef} where $\lambda \in\mathbb{R}$ is such that  $|\lambda|> \frac{N}{R}$, i.e., in particular,  condition \eqref{smallness} fails. We show that this choice of $\lambda$ leads to a contradiction. 

\medskip

In fact, let us consider  a sequence $v_k$ of smooth functions such that 
$$
v_k\to \chi_{B_r}\ \text{in}\ L^1(B_R)\ \ \text{and}\ \ \int_{B_R} |\nabla v_k| \to {\rm Per}(B_r)
$$ 
that is always possible (see for instance \cite[Theorem 3.1]{ls}).

Using $v_k$ to test the equation solved by $u$ one gets
$$
\int_{B_R} z\cdot \nabla v_k =\lambda \int_{B_R} v_k\,;
$$
that is 
$$
|\lambda | \int_{B_R} v_k\leq \|z\|_{L^{\infty}(\Omega)^N}\int_{B_R}|\nabla v_k|\leq \int_{B_R}|\nabla v_k|\,.
$$
Passing to the limit in $k$ one then has
$$
|\lambda |\leq \frac{{\rm Per}(B_r)}{|B_r|}=\frac{N}{r}\,.$$
Due to the arbitrariness of $r$ this latter fact  contradicts the assumption on $\lambda$. 
\end{remark}

\section{The non-autonomous case with a general nonlinearity}\label{4}

This section is devoted to the study  of the Dirichlet problem associated with the mean curvature equation in presence of a general,   possibly singular,  nonlinearity depending on $u$, i.e.  we consider 
\begin{equation}
	\label{pbh}
	\begin{cases}
		\dis -\operatorname{div}\left(\frac{D u}{\sqrt{1+|D u|^2}}\right) = h(u)f & \text{in}\;\Omega,\\
		u=0 & \text{on}\;\partial\Omega,
	\end{cases}
\end{equation}
for a positive  $f$ belonging  to $L^N(\Omega)$. The case of a merely  nonnegative $f$ may also be faced but it  requires some more technical arguments and it  will be  discussed in   Section \ref{sec:fnonnegative} below.

The function $h:[0,\infty)\mapsto [0,\infty]$ is continuous, finite outside the origin,  such that
\begin{equation}\label{h1}
	\begin{aligned}
		\displaystyle &\exists\;{c_1},\gamma,s_1>0\;\ \text{such that}\;\  h(s)\le \frac{c_1}{s^\gamma} \ \ \text{if} \ \ s\leq s_1\,, \ \  \ \ h(0)\not= 0, \\ &\text{and}\ \ 
		h(\infty):=\limsup_{s\to \infty} h(s)<\infty. 
	\end{aligned}
\end{equation}

For later purposes we denote by
\begin{equation}\label{hk}
	h_k(\infty):= \sup_{s\in[k,\infty)} h(s);
\end{equation}
observe that  $h_k(\infty)$ converges to $h(\infty)$ as $k\to \infty$.

Moreover, for the sake of exposition, from here on  we shall make   use of the following notation: 
\begin{equation}\label{sigma}
	\sigma:= \max(1,\gamma).
\end{equation}

\medskip

Although it is known  that the presence of zero order terms of these type produces regularizing effects in similar contexts,    we stress that treating  \eqref{pbh} is strikingly different than dealing with \eqref{pb} as, for instance, a possibly singular $h$ raises the need of a suitable control for the  zones in which the    approximating solutions degenerate.   Secondly, as we will see, the right-hand of the  equation in \eqref{pbh} will be only  locally integrable in general, even if $f$ belongs to $L^N(\Omega)$, bringing  some new technical difficulties. Furthermore,  solutions  need not  possess a trace in the classical sense if the nonlinearity grows too much at zero (i.e. $\sigma>1$).

\medskip

The above discussion makes clear that a particular attention  on the notion of  solution's definition is needed in order to properly  extend  the one of the previous section:
 \begin{defin}\label{defgeneralh}
	Let $ f>0$ a function in  $L^N(\Omega)$. A nonnegative function $u\in BV_{\rm{loc}}(\Omega)$  is a solution to \eqref{pbh} if $h(u)f \in L^1_{\rm loc}(\Omega)$ and if there exists $z\in \DM(\Omega)$ with $||z||_{L^\infty(\Omega)^N}\le 1$ such that	
	\begin{align}
		&	-\operatorname{div}z = h(u)f \ \ \ \text{in  } \mathcal{D'}(\Omega), \label{def_eq}\\
		&(z,D u)=\sqrt{1+|Du|^2} - \sqrt{1-|z|^2} \label{def_zp=1sing} \ \ \ \ \text{as measures in } \Omega,
		\\
		&	\lim_{\epsilon\to 0}  \fint_{\Omega\cap B(x,\epsilon)} u (y) dy = 0 \ \ \ \text{or} \ \ \ [z,\nu] (x)= -1 \label{def_bordoh}\ \ \ \text{for  $\mathcal{H}^{N-1}$-a.e. } x \in \partial\Omega.			
	\end{align}
\end{defin}

\begin{remark}
	Some comments about Definition \ref{defgeneralh} are in order. Firstly observe that the definition does not depend on $\gamma$. 
	Moreover, condition \eqref{def_bordoh} is a weak way to recover the Dirichlet boundary datum which is classical in similar contexts involving, for instance, the $1$-Laplace operator. In particular, let us underline that, in case $h=1$, condition \eqref{def_bordop=1} clearly implies \eqref{def_bordoh}. Obviously, the weaker request  \eqref{def_bordoh} comes from the lack of $BV$-trace for solutions which are, in general, not expected to be well defined  in presence of a strongly singular nonlinearity. 
	\\
	Finally let also underline that, if $h(0)=\infty$, the previous definition implies that $u>0$ almost everywhere in $\Omega$. 
\end{remark}

We begin stating the existence of a solution to \eqref{pbh}.
\begin{theorem}\label{teoexh}
	Let $f\in L^N(\Omega)$ be a positive function  such that \begin{equation}\label{smallnessh}\displaystyle ||f||_{L^N(\Omega)}<\frac{1}{\mathcal{S}_1h(\infty)},
	\end{equation} 
	and let $h$ satisfy \eqref{h1}.
Then there exists a bounded solution $u$ to problem \eqref{pbh} in the sense of Definition \ref{defgeneralh}. Moreover,  $u^{\sigma}\in BV(\Omega)$. 
\end{theorem}

Under proper additional assumptions on $h$, we will also show uniqueness of  bounded solutions to \eqref{pbh}.
\begin{theorem}\label{teounique}
Let  $h$ be decreasing and let $f$ be a positive function in $L^N(\Omega)$. Then there is at most one solution $u\in BV(\Omega)\cap L^\infty(\Omega)$ of  problem \eqref{pbh}.
\end{theorem}

\medskip

\subsection{Proof of Theorem \ref{teoexh}}
 As in the case $h\equiv 1$ we start by considering  the following approximation
\begin{equation}
	\label{pbph}
	\begin{cases}
		\dis \dis -\operatorname{div}\left(\frac{\nabla u_p}{\sqrt{1+|\nabla  u_p|^2}}\right) - (p-1)\operatorname{div}\left(|\nabla u_p|^{p-2}\nabla u_p\right)  = h_p(u_p)f & \text{in}\;\Omega,\\
		u_p=0 & \text{on}\;\partial\Omega,
	\end{cases}
\end{equation}
where $h_p(s):= T_{\frac{1}{p-1}}(h(s))$.
Again, it follows from \cite{ll} the existence of a solution $u_p\in W^{1,p}_0(\Omega)\cap L^\infty(\Omega)$ to \eqref{pbph}. Clearly, $u_p$ is nonnegative since $f$ is positive. We exhibit some basic estimates on $u_p$. Again we understand that  estimates are uniform if  there exists some $p_0>1$ such that the estimate holds uniformly  in the range $1<p\leq p_0$. Recall that $\sigma$ is defined in \eqref{sigma}.

\begin{lemma}\label{lem_stimepriorih}
Let $f \in L^N(\Omega)$ be a positive function such that \eqref{smallnessh} is in force, and let $h$ satisfy \eqref{h1}. Let $u_p$ be a solution of problem \eqref{pbph},  then $u_p$ is locally bounded in $BV(\Omega)$, $u_p^\sigma$ is bounded in $BV(\Omega)$, and $u_p$ is bounded in $L^\infty(\Omega)$ with respect to $p$. As a consequence, $u_p$ converges, up to subsequences, almost everywhere in $\Omega$ to a function $u\in BV_{\rm loc}(\Omega)\cap L^\infty(\Omega)$, in $L^q(\Omega)$ for any $q<\infty$,  and weak$^*$ in $L^\infty(\Omega)$ as $p\to 1^+$. Finally, for any nonnegative $\varphi\in C^1_c(\Omega)$, it holds
\begin{equation}\label{stimapmeno1}
	(p-1)\int_\Omega |\nabla u_p|^p\varphi^{p} \le C, 
\end{equation}
 for  some positive constant $C$ not depending on $p$.   
\end{lemma}
\begin{proof}
	Let us firstly show that $u_p$ is bounded by a constant independent of $p$. We just sketch the calculation since the reasoning is very similar to the one given in the proof of Lemma \ref{lemma_stime}.
	
	\medskip 
	For  $k>0$,  we  consider  $G_k(u_p)$ as a test function in the weak formulation of \eqref{pbph}, yielding to (recall  $h_k(\infty)$  defined as in \eqref{hk})
	\begin{equation}\label{stimeapriorihbounded}
	\begin{aligned}
			\int_\Omega \frac{|\nabla G_k(u_p)|^2}{\sqrt{1+|\nabla G_k(u_p)|^2}} \le  \int_\Omega h_p(u_p)f G_k(u_p) &\le h_k(\infty)||f||_{L^N(\Omega)} ||G_k(u_p)||_{L^{\frac{N}{N-1}}(\Omega)} 
	\\
		&\le h_k(\infty)||f||_{L^N(\Omega)} \mathcal{S}_1 \int_\Omega |\nabla G_k(u_p)|,
	\end{aligned}		
	\end{equation}
having getting rid of the nonnegative term involving the $p$-laplacian and by using the H\"older and the Sobolev inequalities.
	Now, if  $A_k:=\{x\in\Omega: |u_p(x)| > k\}$, we can  write  
	\begin{equation*}\label{stimeapriorihbounded2}
		\int_\Omega \frac{|\nabla G_k(u_p)|^2}{\sqrt{1+|\nabla G_k(u_p)|^2}}  = \int_{A_k} {\sqrt{1+|\nabla G_k(u_p)|^2}}  - \int_{A_k} \frac{1}{\sqrt{1+|\nabla G_k(u_p)|^2}}  \ge \int_\Omega |\nabla G_k(u_p)| - |A_k|,
	\end{equation*}
which, gathered in \eqref{stimeapriorihbounded}, implies that
	\begin{equation*}
		\int_\Omega |\nabla G_k(u_p)| \le \frac{|A_k|}{1-h_k(\infty)||f||_{L^N(\Omega)} \mathcal{S}_1},
	\end{equation*}
	where $k\ge \overline{k}$ for some $\overline{k}>0$ such that $1-h_{\overline{k}}(\infty)||f||_{L^N(\Omega)} \mathcal{S}_1>0$.
	The previous estimate allows to reason as in the proof of Lemma \ref{lemma_stime} in order to conclude that there exists some positive constant $C$ independent of $p$ and such that $||u_p||_{L^\infty(\Omega)}\le C$.

	\medskip
	
	Now we show that $u_p^\sigma$ is bounded in $BV(\Omega)$ with respect to $p$. To this aim we take $u_p^\sigma$ as a test function in the weak formulation of \eqref{pbph} obtaining 
	\begin{equation}\label{stimeapriori1h}
		\sigma\int_{\Omega} \frac{|\nabla u_p|^2u_p^{\sigma-1}}{\sqrt{1+|\nabla u_p|^2}} + \sigma(p-1)\int_{\Omega} |\nabla u_p|^pu_p^{\sigma-1} = \int_\Omega h_p(u_p)f u_p^\sigma.
	\end{equation}
	The right-hand of \eqref{stimeapriori1h} can be estimated as follows 
	\begin{equation}
	\begin{aligned}
	\label{stimeapriori2h}
	\int_\Omega h_p(u_p)f u_p^\sigma &\le c_1s_1^{\sigma-\gamma}\int_{\{u_p< s_1\}} f + h_{s_1}(\infty)\int_{\{u_p\ge s_1\}} fu_p^\sigma \le C,
	\end{aligned}
\end{equation}
where $C$ is a positive constant not depending on $p$ since $u_p$ is bounded in $L^\infty(\Omega)$ with respect to $p$. 

\medskip

For the left-hand of \eqref{stimeapriori1h} one can write
\begin{equation}\label{stimeapriori3h}
	\begin{aligned}
    	\sigma\int_{\Omega} \frac{|\nabla u_p|^2u_p^{\sigma-1}}{\sqrt{1+|\nabla u_p|^2}} + \sigma(p-1)\int_{\Omega} |\nabla u_p|^pu_p^{\sigma-1} &\ge 	\sigma\int_{\Omega} \frac{(1+|\nabla u_p|^2)u_p^{\sigma-1}}{\sqrt{1+|\nabla u_p|^2}} - \sigma\int_{\Omega} \frac{u_p^{\sigma-1}}{\sqrt{1+|\nabla u_p|^2}}
    	\\
    	& \ge \sigma\int_{\Omega}\sqrt{1+|\nabla u_p|^2}u_p^{\sigma-1} -\sigma\int_{\Omega} u_p^{\sigma-1}
    	\\
    	& \ge \int_{\Omega} |\nabla u_p^\sigma| -\sigma\int_{\Omega} u_p^{\sigma-1}.
	\end{aligned}
\end{equation}
Thus, collecting \eqref{stimeapriori2h} and \eqref{stimeapriori3h} in \eqref{stimeapriori1h}, one is lead to
\begin{equation*}\label{stimeapriori4h}
	\int_{\Omega} |\nabla u_p^\sigma| \le C + \sigma\int_{\Omega} u_p^{\sigma-1} \le C, 
\end{equation*}
for a constant $C$ not depending on $p$ since, again,  $u_p$ is bounded and  globally in $BV(\Omega)$ provided $\gamma\le 1$. 

\medskip
Now let us focus on proving that $u_p$ is locally bounded in $BV(\Omega)$ when $\gamma>1$.

\medskip
Let us assume $0\le \varphi\in C^1_c(\Omega)$ and let us take $v=(u_p-||u_p||_{L^\infty(\Omega)})\varphi^p$ to test  \eqref{pbph}. Hence, since $v$ is nonpositive one has
\begin{equation*}
	\begin{aligned}
 		\int_{\Omega} \frac{|\nabla u_p|^2\varphi^p}{\sqrt{1+|\nabla u_p|^2}} &+ p\int_{\Omega} \frac{\nabla u_p\cdot \nabla\varphi}{\sqrt{1+|\nabla u_p|^2}}(u_p-||u_p||_{L^\infty(\Omega)})\varphi^{p-1} + (p-1)\int_{\Omega} |\nabla u_p|^p\varphi^p 
 		\\
 		&+ p(p-1)\int_{\Omega} |\nabla u_p|^{p-2}\nabla u_p \cdot \nabla\varphi(u_p-||u_p||_{L^\infty(\Omega)})\varphi^{p-1} \le 0.
 	\end{aligned}	
\end{equation*}
From the previous inequality one simply gets 
\begin{equation}\label{stimeaprioriTK}
	\begin{aligned}
		&\int_{\Omega} \frac{|\nabla u_p|^2\varphi^p}{\sqrt{1+|\nabla u_p|^2}} + (p-1)\int_{\Omega} |\nabla u_p|^p\varphi^p 
		\\
		&\le p ||u_p||_{L^\infty(\Omega)} \int_{\Omega} |\nabla \varphi|\varphi^{p-1} +   p(p-1)||\nabla \varphi||_{L^\infty(\Omega)^N}||u_p||_{L^\infty(\Omega)}\int_{\Omega} |\nabla u_p|^{p-1}\varphi^{p-1}.
	\end{aligned}	
\end{equation}
Let observe that the Young inequality gives that 
$$\int_{\Omega} |\nabla u_p|^{p-1}\varphi^{p-1} \le \frac{p-1}{p}\int_{\Omega} |\nabla u_p|^{p}\varphi^{p} + \frac{1}{p}|\Omega|,$$
which, gathered in \eqref{stimeaprioriTK}, means that

\begin{equation}\label{stimalocale1}
	\begin{aligned}
		&\int_{\Omega} \frac{|\nabla u_p|^2\varphi^p}{\sqrt{1+|\nabla u_p|^2}} + (p-1)\left(1- ||\nabla \varphi||_{L^\infty(\Omega)^N}||u_p||_{L^\infty(\Omega)}(p-1)\right)\int_{\Omega} |\nabla u_p|^p\varphi^p 
		\\
		&\le p||\nabla \varphi||_{L^\infty(\Omega)^N}||\varphi||_{L^\infty(\Omega)}^{p-1}||u_p||_{L^\infty(\Omega)}|\Omega| +   ||\nabla \varphi||_{L^\infty(\Omega)^N}||u_p||_{L^\infty(\Omega)}|\Omega|(p-1).
	\end{aligned}	
\end{equation}
 Hence it is sufficient requiring $p$ small enough to obtain a nonnegative second term on the left-hand of \eqref{stimalocale1}. Therefore, since we have already shown that $u_p$ is bounded in $L^\infty(\Omega)$ with respect to $p$, we have that
 \begin{equation*}\label{stimalocale2}
 	\begin{aligned}
 		&\int_{\Omega} \frac{|\nabla u_p|^2\varphi^p}{\sqrt{1+|\nabla u_p|^2}} \le C,
 	\end{aligned}	
 \end{equation*}
for some positive constant $C$ which does not depend on $p$. Reasoning similarly to \eqref{stimeapriori3h}, one can prove that 
 \begin{equation*}\label{stimalocale3}
	\begin{aligned}
		&\int_{\Omega} |\nabla u_p|\varphi^p \le C,
	\end{aligned}	
\end{equation*}
namely $u_p$ is locally bounded in $BV(\Omega)$ with respect to $p$.

It is also clear from \eqref{stimalocale1} that, for any $0\le \varphi\in C^1_c(\Omega)$, we have  
\begin{equation*}
	(p-1)\int_\Omega |\nabla u_p|^p\varphi^{p}\le C, 
\end{equation*}
for some positive constant $C$ not depending on $p$.
\medskip

The previous estimates  assure that $u_p$ converges almost everywhere in $\Omega$, up to subsequences,  to a function $u\in BV_{\rm loc}(\Omega)$ as $p\to 1^+$. The $L^\infty$-estimate on $u_p$ gives that the sequence converges to $u$ in $L^q(\Omega)$ for any $q<\infty$ and weak$^*$ in $L^\infty(\Omega)$ as $p\to 1^+$. This concludes the proof.
\end{proof}

\begin{remark}
	Let observe that, if $\gamma\le 1$, $u_p$ is bounded in $BV(\Omega)$  with respect to $p$ and its almost everywhere limit in $p$ belongs to $BV(\Omega)$ as well. On the other hand, in general, the global estimate in $BV(\Omega)$ is only shown for power $\sigma>1$ of $u_p$ which is, obviously,  a weaker statement.
	
	Secondly, we want to highlight that estimate \eqref{stimapmeno1} gives that
\begin{equation}\label{stimapmeno1bis}
	(p-1)\int_\omega |\nabla u_p|^p\le C \ \text{for any }\omega\subset \subset \Omega, 
\end{equation}
where $C$, even depending on $\omega$, does not depend on $p$. This energy estimate will be used  to show the vanishing   of the second term in the approximation scheme as $p\to 1^+$. 
\end{remark}
	We explicitly mention that, from here on, $u$ is the function found in the previous lemma; namely it is (up to subsequences) the almost everywhere limit in $\Omega$ of $u_p$ as $p\to 1^+$. Let us now prove that there exists a vector field $z$ satisfying \eqref{def_eq}; for later purposes we will also gain an extension of the admissible test functions in \eqref{def_eq}.

\begin{lemma}\label{lem_campozh}
	Under the assumptions of Lemma \ref{lem_stimepriorih} there exists $z\in \DM(\Omega)$ with $||z||_{L^\infty(\Omega)^N}\le 1$ and such that 
	\begin{equation}\label{eqdefestesa}
		-\int_\Omega v\operatorname{div}z = \int_\Omega h(u)fv, \ \ \ \forall v \in BV(\Omega)\cap L^\infty(\Omega).
	\end{equation}
\end{lemma}
\begin{proof}
	We first observe that, since $|\nabla u_p|(1+|\nabla u_p|^2)^{-\frac{1}{2}}\le 1$, $\nabla u_p(1+|\nabla u_p|^2)^{-\frac{1}{2}}$ converges weak$^*$ in $L^\infty(\Omega)^N$ to a vector field $z$ as $p\to 1^+$ such that $||z||_{L^\infty(\Omega)^N}\le 1$.

	 We first show that  
	\begin{equation}\label{eqdef}
		-\operatorname{div}z = h(u)f \ \ \text{in}\ \mathcal{D'}(\Omega)\,.
	\end{equation}	

 We then consider a function $\varphi \in C^1_c(\Omega)$. Passing  to the limit in the first term in the weak formulation of \eqref{pbph} is effortless so we focus on the remaining two terms. 
	
	\medskip
	
	Thanks to \eqref{stimapmeno1} (see also \eqref{stimapmeno1bis}) one has 
	\begin{equation}\label{pmeno1azero}
	\begin{aligned}	
		\left|(p-1)\int_\Omega |\nabla u_p|^{p-2}\nabla u_p \cdot \nabla \varphi \right| &\le (p-1)\left((p-1)\int_{\{\supp \varphi\}} |\nabla u_p|^p\right)^{\frac{p-1}{p}}\left(\int_{\{\supp \varphi\}} |\nabla \varphi|^p\right)^{\frac{1}{p}}
	\\\\
	&\le (p-1)C^{\frac{p-1}{p}}||\nabla \varphi||_{L^\infty(\Omega)^N} |\{\supp \varphi\}|^{\frac{1}{p}} \stackrel{p\to1^+}{\longrightarrow} 0.
	\end{aligned}	
\end{equation} 
	
	\medskip
	
	It remains to pass to the limit the right-hand of \eqref{pbph}. If $h(0)<\infty$ one gets no problems using dominated convergence theorem, so that, without losing generality,  let us  assume that $h(0)=\infty$.
	
	\medskip
	
	We first show  that $h(u)f$ is locally integrable. Let  $\varphi\in C^1_c(\Omega)$ nonnegative  to test  \eqref{pbph};  it is clear that
	\begin{equation}\label{stimaL1loc}
		\int_\Omega h_p(u_p)f\varphi\  { = } \int_\Omega \frac{\nabla u_p\cdot \nabla \varphi}{\sqrt{1+|\nabla u_p|^2}} + (p-1)\int_\Omega |\nabla u_p|^{p-2}\nabla u_p\cdot \nabla \varphi \le C,
	\end{equation}  
	where we used both the boundedness of the vector field $\nabla u_p(1+|\nabla u_p|^2)^{-\frac{1}{2}}$ and   \eqref{stimapmeno1bis}. 
	 An application of the Fatou Lemma as $p\to 1^+$ in \eqref{stimaL1loc} gives that 
	\begin{equation}\label{l1loc}
		\int_\Omega h(u)f\varphi \le C,
	\end{equation} 
	that implies the local integrability of  $h(u)f$. We underline that, since $h(0)=\infty$ and $f>0$ almost everywhere in $\Omega$, \eqref{l1loc} also entails that $u>0$ almost everywhere in $\Omega$. 
	
	\medskip
	
	In order to  check the validity of  \eqref{eqdef} let us consider  $V_\delta(u_p)\varphi$ ($0\le \varphi \in C^1_c(\Omega)$ and $V_\delta$ is defined in \eqref{Vdelta}) in the weak formulation of \eqref{pbph}, obtaining
	\begin{equation}\label{stimasing1}
		\begin{aligned}
			\int_{\{u_p\le \delta\}} h(u_p)f\varphi & \le \int_\Omega h(u_p)fV_\delta(u_p)\varphi \le		
		 \int_\Omega \frac{\nabla u_p\cdot \nabla \varphi V_\delta(u_p)}{\sqrt{1+|\nabla u_p|^2}} 
			\\ &+ (p-1)\int_\Omega |\nabla u_p|^{p-2}\nabla u_p\cdot \nabla \varphi V_\delta(u_p)
		\end{aligned}
	\end{equation}	  
	where we have gotten rid of the nonpositive terms involving $V'_\delta$.
	Hence we can take the limsup as $p\to 1^+$ deducing that 
	\begin{equation}\label{stimasing2}
		\begin{aligned}
			&\limsup_{p\to 1^+}\int_{\{u_p\le \delta\}} h_p(u_p)f\varphi  \le 		
			\int_{\Omega} z\cdot \nabla \varphi V_\delta(u)
		\end{aligned}
	\end{equation}
	and the second term on the right-hand of \eqref{stimasing1} goes to zero as $p\to 1^+$ as for   \eqref{pmeno1azero} (recall that $V_\delta(s)\le 1$ for any $s\ge 0$). 
	\\
	Now, it  follows from \eqref{stimasing2} that it holds
	\begin{equation}\label{stimasing3}
	\begin{aligned}
			&\lim_{\delta\to 0}\limsup_{p\to 1^+}\int_{\{u_p\le \delta\}} h_p(u_p)f\varphi  = \int_{\{u=0\}} z\cdot \nabla \varphi \stackrel{u>0}{=} 0.
	\end{aligned}
	\end{equation}
	Estimate \eqref{stimasing3} is the key in order to show the validity of \eqref{eqdef}. Indeed,	let us split as 
	\begin{equation*}
		\begin{aligned}
		 \int_\Omega h_p(u_p)f\varphi = \int_{\{u_p\le \delta\}} h_p(u_p)f\varphi + \int_{\{u_p > \delta\}} h_p(u_p)f\varphi,
		\end{aligned}
	\end{equation*}
where $\delta \not\in \{\eta: |u=\eta|>0\}$ which is a countable set. The first term on the right-hand of the previous vanishes  in, resp.,  $p$ and $\delta$ thanks to \eqref{stimasing3}; the second one, instead,  passes to the limit in $p,\delta$ by two applications of the Lebesgue Theorem since $h(u)f \in L^1_{\rm loc}(\Omega)$. This proves that \eqref{eqdef} holds. Clearly, from the previous arguments, it simply follows the case of $\varphi$ with general sign. Finally observe that the fact that  $z$ is actually in $\DM(\Omega)$ and then the possibility to  extend the set of test functions as stated in \eqref{eqdefestesa} follow from an application of Lemma $5.3$ of \cite{DGOP}. This concludes the proof.
	\end{proof}

Next lemma is about the identification of the vector field emphasized  by  \eqref{def_zp=1sing}.

\begin{lemma}\label{lemma_zh}
Under the assumptions of Lemma \ref{lem_stimepriorih} it holds
\begin{equation}\label{esistenzaz2h}
	(z,Du)=\sqrt{1+|Du|^2} - \sqrt{1-|z|^2} \ \ \text{as measures in } \Omega,
\end{equation}
where $u$ and $z$ are resp.  the function and the vector field given by Lemmas \ref{lem_stimepriorih} and \ref{lem_campozh}. 
\end{lemma}
\begin{proof}
	Let us take $u_p^\sigma\varphi$  as a test function in the weak formulation of \eqref{pbph} where $ \varphi \in C^1_c(\Omega)$ is nonnegative; we get 
	\begin{equation*}
	\begin{aligned}
		\sigma\int_\Omega \frac{|\nabla u_p|^2u_p^{\sigma-1}\varphi}{\sqrt{1+|\nabla u_p|^2}} &+ \int_\Omega \frac{\nabla u_p\cdot \nabla \varphi u_p^{\sigma}}{\sqrt{1+|\nabla u_p|^2}} + (p-1)\sigma\int_\Omega |\nabla u_p|^pu_p^{\sigma-1}\varphi  \\
		&+ (p-1)\int_\Omega |\nabla u_p|^{p-2}\nabla u_p\cdot \nabla \varphi u_p^{\sigma} = \int_\Omega h_p(u_p)fu_p^\sigma\varphi.
	\end{aligned}
	\end{equation*}
 Then, getting rid of the nonnegative third term on the left-hand,  a simple manipulation of  the first term
  yields 
	\begin{equation}\label{stimaperprovaz}
	\begin{aligned}
		\sigma\int_\Omega \sqrt{1+|\nabla u_p|^2}u_p^{\sigma-1}\varphi &- \sigma\int_\Omega \sqrt{1- \frac{|\nabla u_p|^2}{1+|\nabla u_p|^2}} u_p^{\sigma-1}\varphi  
		\le \int_\Omega h_p(u_p)fu_p^\sigma\varphi \\
		&- \int_\Omega \frac{\nabla u_p\cdot \nabla \varphi u_p^{\sigma}}{\sqrt{1+|\nabla u_p|^2}}- (p-1)\int_\Omega |\nabla u_p|^{p-2}\nabla u_p\cdot \nabla \varphi u_p^{\sigma}.
	\end{aligned}
\end{equation}
Now we want to take the liminf,  as $p\to 1^+$,  in the previous inequality.

 Let firstly observe that the first term in \eqref{stimaperprovaz} is nothing else than $\int_\Omega \sqrt{\sigma^2u_p^{2\sigma-2}+ |\nabla u_p^\sigma|^2}\varphi$ which 
is lower semicontinuous with respect to the $L^1$- convergence of $u_p^\sigma$.

On the other hand, the term $$-\int_\Omega \sqrt{1- \frac{|\nabla u_p|^2}{1+|\nabla u_p|^2}} u_p^{\sigma-1}\varphi$$ can be  seen to be lower semincontinuous as $p\to 1^+$. Indeed,  let $F(x)= \sqrt{1-|x|^2}$ which is concave and let $w_p=\frac{\nabla u_p}{\sqrt{1+|\nabla u_p|^2}}$ then
\begin{equation*}
	-\int_\Omega F(w_p)u_p^{\sigma-1} = \int_\Omega F(w_p)\left(u^{\sigma-1}- u_p^{\sigma-1}\right) - \int_\Omega F(w_p)u^{\sigma-1}.
\end{equation*}
The first term on the right-hand of the previous easily goes to zero as $p\to 1^+$ (recall that $|F(x)|\le 1$). While, reasoning  as in the proof of Lemma \ref{lemma_esistenzaz},  one can show that  the second one is weakly lower semicontinuous with respect to the $L^1$ convergence. 

\medskip 

 Now let observe that the first term on the right-hand of \eqref{stimaperprovaz} passes to the limit as $p\to 1^+$ thanks to the convergence of $u_p$ to $u$ in $L^{q}(\Omega)$ for any $q<\infty$ and thanks to the fact that the function  $h(s)s^\sigma$ is bounded.
\\
The second term on the right-hand of \eqref{stimaperprovaz} easily passes to the limit using that  $\nabla u_p (1+|\nabla u_p|^2)^{-\frac{1}{2}}$ converges weak$^*$ in $L^\infty(\Omega)$ to $z$ and $u_p^\sigma$ converges to $u^\sigma$ in $L^{q}(\Omega)$ with $q<\infty$.  Finally, as $u_p$ is uniformly bounded, reasoning similarly  as for  \eqref{pmeno1azero},  the last term tends to zero.

\medskip
Therefore,  gathering together all the previous we  are lead to 
	\begin{equation}\label{pairing1}
	\begin{aligned}
		\sigma\int_\Omega \sqrt{1+|D u|^2}u^{\sigma-1}\varphi - \sigma\int_\Omega \sqrt{1- |z|^2} u^{\sigma-1}\varphi  
		\le \int_\Omega h(u)fu^\sigma\varphi - \int_\Omega z\cdot \nabla \varphi u^\sigma.
	\end{aligned}
\end{equation}

\medskip

Hence,  from \eqref{eqdef},  the fact that $u^\sigma\in BV(\Omega)$, and also using that  $h(u)fu^\sigma \in L^1(\Omega)$, one easily gets that
	\begin{equation}\label{hperu}
	-u^\sigma\operatorname{div}z = h(u)fu^\sigma \ \ \text{in  }\mathcal{D'}(\Omega).
\end{equation}
Then, from \eqref{pairing1} one has
	\begin{equation*}\label{pairing2}
	\begin{aligned}
		\sigma\int_\Omega \sqrt{1+|D u|^2}u^{\sigma-1}\varphi - \sigma\int_\Omega \sqrt{1- |z|^2} u^{\sigma-1}\varphi  
		\le \int_\Omega (z, Du^\sigma)\varphi.
	\end{aligned}
\end{equation*}
For the absolutely continuous part one can reason exactly as in the proof of Lemma \ref{lemma_esistenzaz} if $h(0)<\infty$ (i.e. $\sigma = 1$); otherwise one has $u>0$ almost everywhere in $\Omega$, yielding to 

$$
	z\cdot D^a u = \sqrt{1 +|D^a u|^2} - \sqrt{1-|z|^2}.
$$

 Concerning  the singular part, following again the lines of the proof of Lemma \ref{lemma_esistenzaz},  one gets that,  locally as measures
\begin{equation*}
	(z,Du^\sigma)^s \ge |Du^\sigma|^s.
\end{equation*}
Moreover,  the reverse inequality is trivial since $||z||_{L^\infty(\Omega)^N}\le 1$, then it actually holds
\begin{equation*}
	(z,Du^\sigma)^s = |Du^\sigma|^s.
\end{equation*}
Then, as  $f(t)=t^\sigma$ is increasing,   one can apply Proposition $4.5$ of \cite{crde} deducing that 
\begin{equation*}
	(z,Du)^s = |Du|^s= |D^s u|,
\end{equation*}
locally as measures. This concludes the proof.
\end{proof}

In order to conclude the proof of Theorem \ref{teoexh} it remains to show the following result. 

\begin{lemma}\label{lemma_datoalbordoh}
	Under the assumptions of Lemma \ref{lem_stimepriorih} it holds either
	\begin{equation} \label{bordoh}
	\lim_{\epsilon\to 0}  \fint_{\Omega\cap B(x,\epsilon)} u (y) dy = 0 \ \ \ \text{or} \ \ \ [z,\nu] (x)= -1 \ \ \ \text{for  $\mathcal{H}^{N-1}$-a.e. } x \in \partial\Omega,	
	\end{equation}	
	where $u$ and $z$ are the function and the vector field found respectively in Lemmas \ref{lem_stimepriorih} and \ref{lem_campozh}. 
\end{lemma}
\begin{proof}
	Plug  $u_p^\sigma$ as  test function in the weak formulation of \eqref{pbph}; then after straightforward manipulations, one yields to (recall that $u_p^\sigma$ has zero trace on $\partial\Omega$)
	\begin{equation*}
	\begin{aligned}
		\sigma\int_\Omega \sqrt{1+|\nabla u_p|^2}u_p^{\sigma-1} - \sigma\int_\Omega \sqrt{1- \frac{|\nabla u_p|^2}{1+|\nabla u_p|^2}} u_p^{\sigma-1}  + \int_{\partial\Omega} u_p^\sigma \ d \mathcal{H}^{N-1} 
		\le \int_\Omega h_p(u_p)fu_p^\sigma.
	\end{aligned}
\end{equation*}	
Reasoning as in Lemma \ref{lemma_zh} one can take the liminf as $p\to1^+$ in the previous, obtaining that
	\begin{equation*}
	\begin{aligned}
		\sigma\int_\Omega \sqrt{1+|D u|^2}u^{\sigma-1} - \sigma\int_\Omega \sqrt{1- |z|^2} u^{\sigma-1}  + \int_{\partial\Omega} u^\sigma \ d \mathcal{H}^{N-1} 
		\le \int_\Omega h(u)fu^\sigma.
	\end{aligned}
\end{equation*}	
Hence, thanks to \eqref{hperu}, after applying  the Green formula \eqref{green}, one obtain
\begin{equation*}
	\begin{aligned}
		\sigma\int_\Omega \sqrt{1+|D u|^2}u^{\sigma-1} - \sigma\int_\Omega \sqrt{1- |z|^2} u^{\sigma-1}  + \int_{\partial\Omega} u^\sigma \ d \mathcal{H}^{N-1} 
		\le \int_\Omega (z, D u^\sigma) - \int_{\partial\Omega} [u^\sigma z,\nu] \ d \mathcal{H}^{N-1}.
	\end{aligned}
\end{equation*}
Therefore,  it follows from \eqref{esistenzaz2h} that 
 \begin{equation}\label{cita}
 	\begin{aligned}
 		\int_{\partial\Omega} \left([u^\sigma z,\nu] +u^\sigma \right) \ d \mathcal{H}^{N-1} 
 		\le 0.
 	\end{aligned}
 \end{equation}
Now, using Lemma \ref{lem_campozh},  $z\in \DM(\Omega) $ and, as $u^\sigma \in BV(\Omega)\cap L^{\infty}(\Omega)$, then, by \eqref{des2},   $[u^\sigma z,\nu] = u^\sigma[z,\nu]$.

Recalling that $[z,\nu]\ge -1$ one then deduces by \eqref{cita}    that either $u^\sigma = 0$ or $[z,\nu] = -1$ for $\mathcal{H}^{N-1}$-almost every $x\in \partial\Omega$. Then,  \eqref{bordoh} is obtained as follows: by  \cite[Theorem 3.87]{afp}, if $x\in\partial\Omega$ is such that $u^{\sigma}(x)=0$,  one has				 $$\lim_{\epsilon\to 0} \fint_{\Omega\cap B(x,\epsilon)} u^{\sigma} (y) dy=0\,.$$
If $\sigma>1$,  we can use the  H\"older's inequality to get
$$
		\begin{array}{l}
		 	 		\displaystyle\fint_{\Omega\cap B(x,\epsilon)} u(y) dy \leq \left( \fint_{\Omega\cap B(x,\epsilon)} u^{\sigma} (y) dy\right)^{\frac{1}{\sigma}}\frac{|\Omega\cap B(x,\epsilon)|^{\frac{1}{\sigma'}}}{\epsilon^{\frac{N}{\sigma'}}}\leq C\left( \fint_{\Omega\cap B(x,\epsilon)} u^{\sigma} (y) dy\right)^{\frac{1}{\sigma}}\stackrel{\epsilon\to 0}{\longrightarrow} 0\,;
		 	 		\end{array} 
		 	 		$$
		 	 		that implies \eqref{bordoh}.
\end{proof}

\subsection{Proof of Theorem \ref{teounique}}
Let us prove  the uniqueness result in case of a decreasing lower order term $h$, namely
	\begin{proof}[Proof of Theorem \ref{teounique}]
		Let $u_1$ and $u_2$ be solutions to problem \eqref{pbh} in the sense of Definition \ref{defgeneralh} and let us denote by, respectively, $z_1$ and $z_2$ the corresponding vector fields. 
		We observe that, by Lemma $5.3$ of \cite{DGOP} it follows that both  $h(u_1)f$ and $h(u_2)f$ are in  $L^1(\Omega)$. Then a standard  density argument implies 
		\begin{equation}\label{test}
			-\int_\Omega v\,\operatorname{div} z_i = \int_\Omega  h(u_i)fv, \ \ \forall v\in BV(\Omega)\cap L^\infty(\Omega), \ \ i=1,2.
		\end{equation}	
		 Now we take $v=u_1- u_2$ in the  two weak formulations \eqref{test} related, resp.,  to $u_1$ and $u_2$ and we take the difference. Thus, the application of  \eqref{green} yields  
		\begin{equation*}\label{inequo}
			\begin{aligned}
				\int_\Omega (z_1, Du_1)-\int_\Omega(z_2, Du_1) &+ \int_\Omega(z_2, Du_2) - \int_\Omega(z_1, Du_2) - \int_{\partial\Omega}(u_1-u_2)[z_1,\nu])\, d\mathcal H^{N-1}				    	
				\\
				&+\int_{\partial\Omega}(u_1-u_2)[z_2,\nu])\, d\mathcal H^{N-1}						    	
				= \int_\Omega (h(u_1)- h(u_2))f(u_1- u_2).
			\end{aligned}
		\end{equation*} 
		Let us focus on the second and the fourth term of the previous; we claim that, using \eqref{def_zp=1sing},  one gets that 
		$$\sqrt{1+|Du_1|^2} - \sqrt{1-|z_2|^2}  \ge (z_2,Du_1)$$
		and that
		$$\sqrt{1+|Du_2|^2} - \sqrt{1-|z_1|^2}  \ge (z_1,Du_2)\,,$$
		as measures in $\Omega$. 
		Indeed we proceed  by splitting the measures in the absolutely continuous and singular parts.
		Concerning  the absolutely continuous part of the measures one should prove that
		$$\sqrt{1+|D^a u_1|^2} - \sqrt{1-|z_2|^2}  \ge (z_2,Du_1)^a=z_2\cdot D^a u_1$$
		and that
		$$\sqrt{1+|D^a u_2|^2} - \sqrt{1-|z_1|^2}  \ge (z_1,Du_2)^a=z_1\cdot D^a u_2,$$
		which are purely algebraic inequalities once one recalls that
		$$z_i = \frac{D^a u_i}{\sqrt{1+|D^a u_i|^2}}\ \ \ i=1,2.$$
		On the other hand the  singular part of those  inequalities simply follows  by recalling that $||z_i||_{L^\infty(\Omega)^N}\le 1$.

		As regards the boundary terms,  it follows from \eqref{def_bordoh} that (recall that $u\in BV(\Omega)$) 
		$$
		u_{i}(1+[z_{i},\nu])=0\ \ \mathcal H^{N-1}-\text{a.e. on}\ \partial\Omega\  \ \text{for}\ \  i=1,2.
		$$
		This takes to
		\begin{align*}
			\int_\Omega (h(u_1)- h(u_2))f(u_1- u_2)
			\ge \int_{\partial\Omega}(u_1 +u_1 [z_2,\nu]  )\, d\mathcal H^{N-1}				    	
		 +\int_{\partial\Omega}(u_2[z_1,\nu] + u_2)\, d\mathcal H^{N-1}. 					    		\end{align*}				Again, observing that  $[z_i,\nu] \in [-1,1]$ for $i=1,2$,  the right-hand of the previous is nonnegative. This gives that 
		$$\int_\Omega (h(u_1)- h(u_2))f(u_1- u_2)\ge 0,$$
		which implies $u_1=u_2$ a.e. in $\Omega$ since $f>0$ a.e. in $\Omega$.
	\end{proof}

\section{Extensions,   examples and remarks}
\label{5}
\subsection{The case of a nonnegative source $f$}
\label{sec:fnonnegative}

The previous results can be extended to the case $f$ being purely nonnegative. As we will point out, if $h(0)<\infty$, then the arguments of Section \ref{4} work as well with straightforward minor modifications. 
If   $h(0)=\infty$,  one can not deduce in general that the solution to \eqref{pbh} is positive almost everywhere in $\Omega$; though one can include this case in the theory  by re-adapting an idea  of \cite{DGOP}.   Let us briefly explain how. 
We define the following function
$$
\Psi(s):=
\begin{cases}
	1 \ &\text{if }h(0)<\infty,\\
	\chi_{\{s>0\}} &\text{if }h(0)=\infty.
\end{cases}
$$
One can suitably modify the notion of solution as follow:
 \begin{defin}\label{defnonnegative}
	 A nonnegative function $u\in BV_{\rm{loc}}(\Omega)\cap L^{\infty}(\Omega)$  having $u^\sigma \in BV(\Omega)$ and $\Psi(u) \in BV_{\rm{loc}}(\Omega)$ is a solution to \eqref{pbh} if $h(u)f \in L^1_{\rm loc}(\Omega)$ and if there exists $z\in \DM_{\rm loc}(\Omega)$ with $||z||_{L^\infty(\Omega)^N}\le 1$ such that	
	\begin{align}
		&	-\Psi^*(u)\operatorname{div}z = h(u)f \ \ \ \text{in  } \mathcal{D'}(\Omega), \label{def_eqnon}\\
		&(z,D u)=\sqrt{1+|Du|^2} - \sqrt{1-|z|^2} \label{def_zp=1singnon} \ \ \ \ \text{as measures in } \Omega,
		\\
		&	u^\sigma(x) + [u^\sigma z,\nu] (x)= 0 \label{def_bordohnon}\ \ \ \text{for  $\mathcal{H}^{N-1}$-a.e. } x \in \partial\Omega.			
	\end{align}
\end{defin}
Let us stress again that the previous definition coincides with  
Definition \ref{defgeneralh} both in case $h(0)<\infty$ and in case $f>0$ (even if $h(0)=\infty$). 

\medskip

Apart from the distributional formulation the real novelty in Definition \ref{defnonnegative} is given by the boundary datum. In fact, in this case we do not know if, in general,  $z\in \DM(\Omega)$;  in order  to recover the boundary datum we need to impose the weaker version of \eqref{def_bordoh} given by \eqref{def_bordohnon} as $u^\sigma z \in \DM(\Omega)$ (see Lemma \ref{21}).

\medskip

Keeping in mind these facts  one can show  the following to hold; its proof relies on an effortless re-adaptation of the proof of \cite[Theorem 6.4]{DGOP}. 
\begin{theorem}
	Let $f\in L^N(\Omega)$ be nonnegative and such that \begin{equation}\label{smallnesshi}\displaystyle ||f||_{L^N(\Omega)}<\frac{1}{\mathcal{S}_1h(\infty)},
\end{equation} 
and let $h$ satisfy \eqref{h1}. 
	Then there exists a  solution to problem \eqref{pbh} in the sense of Definition \ref{defnonnegative}.
\end{theorem}

\subsection{The datum $f$ in the critical Marcinkiewicz space $L^{N,\infty}(\Omega)$}
\label{52}

All the results of the previous sections can also  be suitably extended to the case of a general nonnegative datum $f$ belonging to the Marcinkiewicz space $L^{N,\infty}(\Omega)$. Let us refer the interested reader to the monograph \cite{PKF} for an introduction on basic properties of these spaces.    
Here we just mention the following Sobolev embedding  inequality:
$$
	\|v\|_{L^{p^{\ast},p}(\Omega)}\leq \tilde{\mathcal S}_{p} \|\nabla v\|_{W^{1,p}_{0}(\Omega)}, \ \ \forall v\in W^{1,p}_{0}(\Omega),
$$
where
$$ \tilde{\mathcal S}_{p}=\frac{p\Gamma(1+\frac{N}{2})^{\frac{1}{N}}}{\sqrt{\pi}(N-p)} \stackrel{p\to 1^{+}}{\longrightarrow}  \tilde{\mathcal S}_{1}= [(N-1)\omega_{N}^{\frac{1}{N}}]^{-1}, \ \ $$
and with $\Gamma$ we denoted the Gamma function.

\medskip 

With a suitable re-adaptation of the arguments of the previous sections, one could prove the following result:

\begin{theorem}\label{tuttolor}
	Let $0\le f\in L^{N,\infty}(\Omega)$ such that \begin{equation}\label{sharp}\displaystyle ||f||_{L^{N,\infty}(\Omega)}<\frac{1}{\tilde{\mathcal{S}}_1 h(\infty)}\,,\end{equation}
	and let $h$ satisfy \eqref{h1}. Then there exists a bounded \bk solution $u$ to problem \eqref{pbh} in the sense of Definition \ref{defnonnegative}. 	
\end{theorem}

\subsection{An explicit example}

We want to highlight the sharpness  of the bound \eqref{sharp} in order to obtain bounded solutions. For the sake of exposition we choose $\Omega = B_1(0)$ and  $h\equiv 1$. \label{533}
\begin{example}\label{53}
 Let $N\geq 3$ and let $0<\alpha < N-1$; a straightforward computation gives that a radial solution to problem 
 
$$
\label{pbie}
\begin{cases}
\dis -\operatorname{div}\left(\frac{D u_\alpha }{\sqrt{1+|D u_\alpha|^2}}\right) = \frac{N-1}{|x|} g_\alpha (|x|)=:f_\alpha(x)& \text{in}\; B_1(0),\\
u_\alpha=0 & \text{on}\;\partial B_1(0)\,
\end{cases}
$$
is given, if $|x|=r$,  by $u_\alpha (x)=r^{-\alpha} - 1$ provided $g_\alpha:(0,1)\mapsto \mathbb{R}^+$ is given by 

$$
g_{\alpha}(r)=\frac{\alpha r^{-\alpha-1}(\alpha^2 r^{-2\alpha -2}- \frac{\alpha+2 - N}{N-1})}{(1+\alpha^2 r^{-2\alpha-2})^{\frac{3}{2}}}\,.
$$

Observe that both
 $$|g_\alpha (r)|\leq 1\ \ \text\ \ \forall\ 0<\alpha<N-1\,,$$ and 
 $$
 \lim_{r\to 0^+} g_\alpha (r)=1\,.
 $$
 
It follows that
 $$
\left\|f_\alpha(x)\right\|_{L^{N,\infty}(B_1(0))}  = (N-1)\omega_{N}^{\frac{1}{N}}\,. 
$$ 
In fact, fix $t$ and consider $0\leq r_1(t)\leq 1$ such that $r_1= \frac{(N-1)}{t}g_\alpha (r_1)$ (which exists at least for $t>>1$). Then one observes that  $\{|f_\alpha(x)|> t\}$ is a ball (as $g$ decreases) and in particular $$ \{|f_\alpha(x)|> t\}=\{|x|\leq r_1(t)\}.$$ So that 
$$
\begin{array}{l}
t |\{|f_\alpha(x)|> t\}|^{\frac{1}{N}}= t(r_1(t)^N \omega_N)^{\frac{1}{N}}= t \left(\frac{(N-1)}{t}g_\alpha (r_1)\right)\omega_N^{\frac{1}{N}}=(N-1)g_\alpha(r_1 (t))\omega_N^{\frac{1}{N}}\,,
\end{array}
$$
and, observing that $r_1(t)\to 0$ as $t\to \infty$  (we use that  $g_\alpha$ is bounded) taking the supremum on $t$ we get 
 $$
\left\|f_\alpha(x)\right\|_{L^{N,\infty}(B_1(0))}  = (N-1)\omega_{N}^{\frac{1}{N}}\,. 
$$

In other words, unbounded solutions can be found in the extremal case emphasizing the sharpness of \eqref{sharp}  in order to obtain bounded solutions.

\end{example}

\section*{Acknowledgements}
The authors are partially supported by the Gruppo Nazionale per l’Analisi Matematica, la Probabilità e le loro Applicazioni (GNAMPA) of the Istituto Nazionale di Alta Matematica (INdAM).

\section*{Conflict of interest declaration}

The authors declare no competing interests.

\section*{Data availability statement}
 We do not analyse or generate any datasets, because our work
proceeds within a theoretical and mathematical approach. One
can obtain the relevant materials from the references below.

\end{document}